\documentclass[a4paper,twoside]{article}

\usepackage[backref]{hyperref}   %works out of the box for pdflatex
\usepackage{amsrefs}

\usepackage{amsmath,amsthm,amsfonts,latexsym,amscd,amssymb,enumerate}

\swapnumbers
\theoremstyle{plain}
\newtheorem{theorem}{Theorem}[section]
\newtheorem{lemma}[theorem]{Lemma}

\newtheorem{proposition}[theorem]{Proposition}

\theoremstyle{definition}
\newtheorem{definition}[theorem]{Definition}
\newtheorem{example}[theorem]{Example}

\theoremstyle{remark}
\newtheorem{remark}[theorem]{Remark}

\newcommand{\reals}{\mathbb{R}}

\newcommand{\naturals}{\mathbb{N}}
\newcommand{\integers}{\mathbb{Z}}

\newcommand{\boundary}[1]{\partial#1}

\newcommand{\tensor}{\otimes}
\newcommand{\into}{\hookrightarrow}

\newcommand{\iso}{\cong}
\newcommand{\disjointunion}{\amalg}
\newcommand{\union}{\cup}

\DeclareMathOperator{\sing}{sing}
\DeclareMathOperator{\Hom}{Hom}    %Homomorphisms

\DeclareMathOperator{\pr}{pr}
\DeclareMathOperator{\characteristic}{char}

\DeclareMathOperator{\TC}{TC}
\DeclareMathOperator{\wgt}{wgt}

\newcommand{\forget}[1]{}

{\catcode`@=11\global\let\c@equation=\c@theorem}

\allowdisplaybreaks[2]

\begin{document}
\pagestyle{myheadings}
\markboth{Luca Sandrock and Thomas Schick}{Topological Complexity of Symplectic CW-complexes}

\title{Topological Complexity of
  symplectic CW-complexes}

\author{Luca
  Sandrock\thanks{\protect\href{mailto:l.sandrock01@stud.uni-goettingen.de}{email:
      l.sandrock01@stud.uni-goettingen.de}} \ and Thomas Schick\thanks{
\protect\href{mailto:thomas.schick@math.uni-goettingen.de}{e-mail:
  thomas.schick@math.uni-goettingen.de}
\protect\\
\protect\href{https://www.uni-math.gwdg.de/schick}{www:~https://www.uni-math.gwdg.de/schick}}\\
Mathematisches Institut\\
Georg-August-Universit\"at G{\"o}ttingen\\
Germany}

\date{}
\maketitle

\begin{abstract}
  A cohomology class $u$ of a topological space $X$ is atoroidal if its pullback to the torus vanishes for every
  map from a torus to $X$. Furthermore, $X$ is  atoroidally symplectic if there is an atoroidal
  cohomology class $u\in H^2_{\sing}(X;F)$ such that $u^n \ne 0$.
  We prove that every atoroidally symplectic CW-complex $X$ of dimension $2n$ has topological complexity $4n$.
  This generalizes a result of Grant and Mescher who prove the corresponding statement in the case where $X$ is an
  atoroidally c-symplectic manifold and $u$ is a de Rham cohomology class.
  Using this generalisation, we obtain new calculations of topological complexity, including for many
  products of 3-manifolds and of group presentation complexes.
\end{abstract}

\section{Introduction}

Motion planning is a computational problem which is concerned with finding transitions between configurations and
frequently occurs in applications, such as in path planning for robot motions.
Topological complexity is a numerical homotopy invariant, originally defined in \cite{Farber}, which gives a rough,
but rigorous, lower bound on its complexity. To formalize this, let $X$ be a path-connected topological space (of
configurations) and
\begin{equation*}
  PX:= C([0,1],X)
\end{equation*}
the space of continuous paths in $X$ endowed with compact-open topology and let $\pi$ be the fiber bundle projection
\begin{equation*}
  \pi\colon PX \to X\times X; \; \gamma\mapsto (\gamma(0),\gamma(1)).
\end{equation*}

The \emph{topological complexity} of $X$, denoted $\TC(X)$, is the minimal
number $k$ of sets forming an open cover $\{U_i\}_{i=0,\dots, k}$ of
$X\times X$ such that for each $j\in\{0,\dots,k\}$ there exists a continuous
map $s_j\colon U_j\to PX$ with $\pi\circ s_j = \iota_j\colon U_j\into X\times
X$. We set $\TC(X):=\infty$ if no such $k$ exists.

The topological complexity has an easy upper bound in terms of dimension and
connectivity \cite[Theorem 5.2]{Farber_instabilities}: If $X$ is an
$(r-1)$-connected CW-complex (i.e.~every map $f\colon S^j\to X$ for $j<r$ is
null-homotopic) then
\begin{equation*}
  \TC(X)\le 2\frac{\dim(X)}{r}.
\end{equation*}
Of interest in this note is the case of a non-simply connected space, i.e.~$r=1$,
and we look for conditions which give a matching lower bound on the
topological complexity. 

This invariant has received considerable attention in the last decades. The
inspiration for this short note is the article \cite{GrantMescher} of Grant
and Mescher. In contrast to their work, we avoid the use of de Rham cohomology
on infinite dimensional manifolds and instead develop tools which use elementary
singular (co-)homology which allows us to generalize their main result from the
context of smooth manifolds to the one of arbitrary CW-complexes. An additional
benefit is that one now can also work with more general coefficient rings, further
generalizing the result. Along the way, we correct a small glitch in a Mayer-Vietoris
argument of \cite{GrantMescher} for the fiberwise join $P_2X$.

Some of the results of this thesis form part of Luca Sandrock's
bachelor thesis \cite{Sandrock} written at Universit\"at G\"ottingen.

The starting point of \cite{GrantMescher} are the following two definitions:

\begin{definition}
  Let $X$ be a topological space, $A$ a local coefficient system,
  $k\in\naturals$ and $u\in 
  H^k(X;A)$ a cohomology class. This class is called \emph{atoroidal} if
  $f^*u=0$ for every map $f\colon T^k\to X$.
\end{definition}
\begin{definition}
  A connected smooth manifold $M$ of dimension $2n$ is called c-symplectic if
  it is equipped with a closed 2-form $\omega\in\Omega^2(M)$ such that
  $[\omega]^n \ne 0\in H^{2n}_{dR}(M)$. Note that every closed symplectic
  manifold is also c-symplectic.

  A c-symplectic manifold $(M,\omega)$ is called \emph{atoroidally c-symplectic}
  if $\omega$ is atoroidal. In the given situation this is equivalent to
  \begin{equation*}
    \int_{T^2}f^*\omega =0 \quad\text{for every smooth map } f\colon T^2\to M.
  \end{equation*}
\end{definition}

The main result of \cite{GrantMescher} is its Theorem 1.2:
\begin{theorem}
  Let $(M,\omega)$ be an atoroidally c-symplectic manifold. Then
  \begin{equation*}
    \TC(M)= 2\dim(M).
  \end{equation*}
\end{theorem}

We generalize this result to arbitrary CW-complexes and cohomology with
general coefficients.

\begin{theorem}\label{theo:main}
  Let $X$ be a connected CW-complex of dimension $2n$ and let $F$ be a
  coefficient field with $\characteristic(F)=0$ or
  $\characteristic(F)>2n$. Assume that 
  there exists an atoroidal cohomology class
  $u\in H^2(X;F)$ such that $u^n \ne 0 \in
  H^{2n}(X;F)$.

  Then  $\TC(X) = 2\dim(X)$.
\end{theorem}

We also have a variant of the theorem which is suitable in particular for
products of 2-dimensional aspherical complexes:
\begin{theorem}\label{theo:specialmain}
  Let $X$ be a connected CW-complex of dimension $2n$ and let $F$ be a coefficient field with
  $\characteristic(F)\ne 2$. Assume there exist atoroidal cohomology classes $u_1,\dots,u_n\in H^2(X;F)$ such
  that $u_j^2=0$ for $j=1,\dots,n$ and $u_1\cup\dots\cup u_n\ne 0\in
  H^{2n}(X;F)$.

  Then  $\TC(X)=2\dim(X)$.
\end{theorem}

Following \cite{GrantMescher}, the proof of Theorem \ref{theo:main} and
Theorem \ref{theo:specialmain} uses the
concept of $\TC$-weight of cohomology classes introduced in
\cite{FarberGrant}. Here, $u\in H^*(X\times X;F)$ has $\wgt(u)\ge 2$ if and
only if it is in the kernel of the homomorphism induced from the fiberwise join
$p_2\colon P_2X\to X\times X$ of the free path fibration with itself,
introduced below in Definition \ref{def:LX_PX}. The key
point is that
\begin{equation*}
  u\ne 0\; \implies \;\TC(X)\ge \wgt(u).
\end{equation*}

Our main technical result is that for $u\in H^2(X;F)$ the associated
zero-divisor
\begin{equation*}
  \overline u:= 1\times u-u\times 1\in H^2(X\times X; F)
\end{equation*}
has
$\wgt(\overline u)=2$.

\section{Cohomological Lower Bounds for the Topological Complexity}

In this section, we recall the concept of $\TC$-weight introduced in
\cite{FarberGrant} and defined as follows.
\begin{definition}
  Let $X$ be a CW-complex and $A$ a local coefficient system on $X\times
  X$. The $\TC$-weight $\wgt(u)\in \{0,1,\dots\}\union\{\infty\}$ of a class
  $u\in H^*(X\times X;A)$ is the largest $k$ such that $f^*u=0\in H^*(Y;f^*A)$
  for every continuous map $f\colon Y\to X\times X$ for which there exists an
  open covering $\{U_1,\dots,U_k\}$ of $Y$ and continuous maps $f_i\colon
  U_i\to PX$ with $\pi\circ f_i=f|_{U_i}$ for $i=1,\dots,k$.
\end{definition}

The following result from \cite[Proposition 2]{FarberGrant} allows to use the $\TC$-weight efficiently to get
lower bounds on the topological complexity:

\begin{theorem}\label{theo:wgt_rules}
  Let $X$ be a CW-complex with local coefficient systems $A,B$ on $X\times
  X$. Assume that $u_A\in H^*(X\times X;A)$ and $u_B\in H^*(X\times X;B)$ are
  cohomology classes.

  Then  $u_A\cup u_B \in H^*(X\times X; A\tensor B)$ satisfies
  \begin{equation*}
    \wgt(u_A\cup u_B) \ge \wgt(u_A) + \wgt(u_B).
  \end{equation*}

  Moreover, if $u_A\ne 0$ then $\TC(X)\ge \wgt(u_A)$.
\end{theorem}

A criterion for the $\TC$-weight to be (at least) 2 is derived in
\cite{GrantMescher}. It uses the fiberwise join $P_2X$ of two copies of $\pi \colon PX\to X\times
X$ and has the following explicit definition.

\begin{definition}\label{def:LX_PX}
  Let $X$ be a CW-complex with path space $PX$ and end point map $\pi\colon
  PX\to X\times X$. Define the free loop space
\begin{equation*}
  LX:= \{\gamma\in PX \mid \gamma(0)=\gamma(1)\}
\end{equation*}
  and for $j=1,2$ the continuous maps
  \begin{equation}\label{eq:def_of_rj}
    r_j\colon LX\to PX; \;\gamma\mapsto \gamma|_{I_j}\circ \psi_j
  \end{equation}
  where $I_1:= [0,\frac{1}{2}]$ and $I_2:=[\frac{1}{2},1]$ are the left and
  right half interval of $[0,1]$ and $\psi_j\colon [0,1]\to I_j$ are the 
  affine linear homeomorphisms $\psi_1(t):=\frac{t}{2}$ and
  $\psi_2(t):=1-\frac{t}{2}$. Note that $\psi_1$ preserves the orientation
  whereas $\psi_2$ reverses it.

  One now defines $P_2X$ as the double mapping cylinder of $r_1$ and $r_2$, i.e.
  \begin{equation*}
    P_2X:= \big((LX\times [1,2])\disjointunion (PX\times \{1\}) \disjointunion (PX\times \{2\})\big)/\sim
  \end{equation*}
  with equivalence relation $\sim$ generated by $(\gamma,1)\sim
  (r_1(\gamma),1)$ and $(\gamma,2)\sim (r_2(\gamma),2)$ for $\gamma\in
  LX$. We get two canonical inclusions $\iota_j\colon PX\to P_2X;
  \gamma\mapsto [(\gamma,j)]$ for $j=1,2$.

  Since $\pi\circ r_1=\pi\circ r_2$ for the endpoint projection
  $\pi\colon PX\to X\times X$, this map descends to a well defined map
  \begin{equation*}
    \pi_2\colon P_2X\to X\times X; \; [\gamma,t]\mapsto
    \begin{cases}
      (\gamma(0),\gamma(\frac{1}{2})) ;& \gamma\in LX\\
      (\gamma(0),\gamma(1)); & \gamma\in PX
    \end{cases}
    .
  \end{equation*}

\end{definition}

The following \cite[Proposition 2.3]{GrantMescher} is the promised criterion for $\TC$-weight.
\begin{proposition}\label{prop:TC_crit}
  Let $A$ be a local coefficient system on $X\times X$ and $u\in H^*(X\times
  X; A)$. 
  \begin{equation*}
    \text{If }\pi_2^*u = 0 \in H^*(P_2X; \pi_2^*A)\text{ then }\wgt(u)\ge 2.
  \end{equation*}
\end{proposition}

\section{Mayer-Vietoris on the Fiberwise Join}

In this section, we obtain information about the cohomology of $P_2X$ using a
Mayer-Vietoris decomposition. This follows the approach in \cite[Section
3]{GrantMescher} but simplifies the arguments and corrects a glitch in the
proof of \cite[Theorem 3.5]{GrantMescher}.

To obtain the Mayer-Vietoris sequence, we start with the obvious projection
$q\colon P_2X\to [1,2]; [\gamma,t]\mapsto t$, then set
\begin{equation*}
  Z(r_1):= q^{-1}\left(\left[1, \frac{3}{2}\right]\right)\subset P_2X,
  \qquad
  Z(r_2):= q^{-1}\left(\left[\frac{3}{2},2\right]\right)\subset P_2X
\end{equation*}
and use the
decomposition $P_2X = Z(r_1)\cup Z(r_2)$ with intersection $Z(r_1)\cap Z(r_2)
= LX\times \{\frac{3}{2}\}$ which we identify with $LX$. Note that $Z(r_1)$ and
$Z(r_2)$ indeed are the mapping cylinders of $r_1$ and $r_2$, respectively.
We denote with $\delta_{MV}$ the boundary map of the Mayer-Vietoris sequence
\begin{equation*}
  \xrightarrow{} H^1(P_2X;F) \xrightarrow{i_1^*\oplus i_2^*}
  H^1(Z(r_1);F)\oplus H^1(Z(r_2);F) \xrightarrow{ \iota_1^*
  -\iota_2^*} H^1(LX;F) \xrightarrow{\delta_{MV}} H^2(P_2X;F) \xrightarrow{}
\end{equation*}
associated to this decomposition.
Our main technical result then is the following.

We denote with $\delta_{MV}\colon H^1(LX;F) \to H^2(P_2X;F)$ the boundary map
of the Mayer-Vietoris sequence associated to this decomposition.

Our main technical result is then the following.
\begin{proposition}\label{prop:MV_cocycle}
  Let $X$ be a connected CW-complex, $A$ an abelian coefficient group. Let $c\in
  C^2_{\sing}(X;A)$ be a singular cocycle (i.e.~$\boundary(c)=0$ for the
  singular cochain boundary map $\partial$). Set
  \begin{equation*}
    \overline c:= \pr_2^*c-
    \pr_1^*c \in C^2_{\sing}(X\times X; A)
  \end{equation*}
  and assume that there is a cochain $b_c\in
  C^1_{\sing}(PX;A)$ such that $\partial(b_c) = \pi^*\overline c$.

  Then $a_c:= r_1^*b_c-r_2^*b_c\in C^1_{\sing}(LX;A)$ is a cocycle and
  \begin{equation*}
    \delta_{MV}([a_c]) = \pi_2^*([\overline c]) \in H^2(P_2X;A).
  \end{equation*}
\end{proposition}
\begin{proof}
  First observe that $\pi\circ r_1=\pi\circ r_2$, hence by naturality
  \begin{equation*}
    \boundary(a_c) =r_1^*(\boundary b_c)-r_2^*(\boundary b_c) =
    r_1^*\pi^*\overline c - r_2^*\pi^*\overline c =0.
  \end{equation*}

  To compute the Mayer-Vietoris boundary map, we use its snake lemma
  definition applied to the Mayer-Vietoris short exact sequence of singular cochain groups
  {\small
    \begin{equation*}
  0 \to C^*_{\sing}(P_2X;A) \xrightarrow{i_1^*\oplus i_2^*}
  C^*_{\sing}(Z(r_1);A)\oplus C^*_{\sing}(Z(r_2);A) \xrightarrow{ \iota_1^*
    -\iota_2^*} C^*_{\sing}(LX;A) \to 0.
\end{equation*}
}

Specifically, given the cocycle $a_c\in C^1_{\sing}(LX;A)$, we need to find
  cochains $(\beta_1,\beta_2)\in C^1_{\sing}(Z(r_1);A)\oplus
  C^1_{\sing}(Z(r_2);A)$ such that $\iota_1^*\beta_1 - \iota_2^*\beta_2 =
  a_c$, where $\iota_k\colon LX\to Z(r_k)$ are the
  inclusion maps as end of the mapping cylinders.

  In our case, we let $\pr_k\colon Z(r_k)\to PX$ be the projection to the
  other end of the mapping cylinder (the standard homotopy equivalence) and
  use the cocycles
  \begin{equation*}
(\pr_1^*b_c,\pr_2^*b_c)\in C^1_{\sing}(Z(r_1);A)\oplus
  C^1_{\sing}(Z(r_2);A).
\end{equation*}
As we have $\pr_k\circ\iota_k = r_k$, by naturality we indeed get
\begin{equation*}
\iota_1^*\pr_1^*b_c- \iota_2^*\pr_2^*b_c= r_1^*b_c-r_2^*b_c = a_c.
\end{equation*}

We next have to compute the coboundary of $\pr_k^*b_c$. By naturality of the
coboundary map we obtain
  \begin{equation*}
    \boundary(\pr_k^*b_c) = \pr_k^*(\boundary b_c) = \pr_k^*\pi^*\overline c =
    i_k^*(\pi_2^* \overline c).
  \end{equation*}
Here, $i_k\colon Z(r_k)\to P_2X$ is the natural inclusion and the final
equality follows because upon composition with the endpoint projection to
$X\times X$, all relevant maps become equal: $\pi_2 \circ i_k = \pi\circ
\pr_k\colon Z(r_k)\to X\times X$. Now recall that $i_2^*\oplus i_2^*$ is the
inclusion map of the short exact sequence of singular cochain groups
{\small
  \begin{equation*}
  0 \to C^*_{\sing}(P_2X;A) \xrightarrow{i_1^*\oplus i_2^*}
  C^*_{\sing}(Z(r_1);A)\oplus C^*_{\sing}(Z(r_2);A) \xrightarrow{ \iota_1^*
    -\iota_2^*} C^*_{\sing}(LX;A) \to 0.
\end{equation*}
}
We just saw that
\begin{equation*}
  (i_1^*\oplus i_2^*) (\pi_2^* \overline c) =\boundary(\pr_1^*b_c,\pr_2^*b_c),
\end{equation*}
hence by the snake lemma indeed
\begin{equation*}
  \delta_{MV}([a_c]) =\pi_2^*[\overline c]
\end{equation*}
which is what we have to prove.
\end{proof}

\begin{remark}
  The glitch in the proof of \cite[Theorem 3.5]{GrantMescher} we mentioned
  above is the
  application of the ``snake lemma'' to a sequence of singular cochain
  complexes
  {\small
  \begin{equation*}
    0 \to C^*_{\sing}(P_2M;\reals)\xrightarrow{i_1^*\oplus i_2^*}
    C^*_{\sing}(\mathcal{P}M;\reals)\oplus C^*_{\sing}(\mathcal{P}M;\reals)
    \xrightarrow{r_1^*-r_2^*} C^*_{\sing}(\Lambda M;\reals)\to 0
  \end{equation*}
  }
  which however is \emph{not} an exact sequence of cochain complexes but only chain
  homotopy equivalent to one. In general, the calculation of the boundary map
  of the associated long exact sequence via homotopy equivalent replacements
  is not guaranteed to work. With a stroke of luck, in our case the idea
  works however, since the relevant classes pull back from $M\times M$ and
  the chain homotopy equivalences are compatible with the maps to $M\times M$.
\end{remark}

\section{The TC-Weight of an atoroidal Cohomology Class}

We now want to combine Propositions \ref{prop:TC_crit} and
\ref{prop:automatic_atoroidal} to bound the TC-weight of our atoroidal classes
from below. As a tool, we transport singular (co-)chains between a
space and its path space using the exponential adjunction, whose properties we
first collect in the following subsection. These results should be well known
and are elementary, we supply them for completeness of the exposition.

\subsection{Singular Chains of a Space and its Path Space}

\begin{definition}\label{def:curring}
  Let $X$ be a topological space and $A$ a coefficient group. We define the transformation
  \begin{equation*}
    b\colon C^*_{\sing}(X;A) \to C_{\sing}^{*-1}(PX;A); c\mapsto b_c
  \end{equation*}
  setting 
  \begin{equation*}
    b_c:= c\circ T\colon C_1^{\sing}(PX)\xrightarrow{T} C_2^{\sing}(X)\xrightarrow{c} A,
  \end{equation*}
  where $T$ is defined via the exponential law as follows.
  Let
  \begin{equation*}
    \Psi\colon C(\Delta_k, PX) \to C(\Delta_k\times [0,1],X) 
  \end{equation*}
  be the inverse map of the tensor-hom adjunction given by
  \begin{equation*}
    \Psi f (x,s):= (\sigma(x))(s);\qquad \forall \sigma\colon \Delta_k\to PX,\,x\in \Delta_k,\, s\in [0,1].
  \end{equation*}
  Here, $\Delta_k$ is the standard k-simplex spanned by vertices $v_0,\dots,v_k$.
  For $j=0,\dots,k$ let $\iota_j\colon \Delta_{k+1}\to \Delta_k\times [0,1]$
  be the standard embeddings decomposing $\Delta_k\times [0,1]$ 
  into $(k+1)$-dimensional simplices. Explicitly, these are the affine linear maps with
  \begin{equation}\label{eq:decomp_incl}
    \iota_j(v_i) =
    \begin{cases}
      (v_i,0); & i\le j\\
      (v_{i-1},1); & i>j.
    \end{cases}
  \end{equation}
  We then define
  \begin{equation*}
    T(\sigma\colon \Delta_1\to PX):= \Psi(\sigma)\circ \iota_0  -
    \Psi(\sigma)\circ \iota_1.
  \end{equation*}
\end{definition}

A crucial property is the relation between the singular coboundary maps and
the transformation $b$, which reads as follows:

\begin{lemma}
  In the situation of Definition \ref{def:curring}, let $\pi\colon PX\to
  X\times X$ be the endpoint evaluation map. For $c\in
  C^{k}_{\sing}(X;A)$ then
  \begin{equation}\label{eq:boundary_and_b}
   \boundary (b_c) =-  b_{\boundary c}  + \pi^*\pr_2^*c -\pi^*\pr_1^*c.
 \end{equation}
  Define $\overline
  c:=\pr_2(c) - \pr_1(c)\in C^2_{\sing}(X\times X;A)$. If $c$ is a cocycle,
  i.e.~$\boundary c=0$, we then get as short version of Equation
  \eqref{eq:boundary_and_b} that
  \begin{equation*}
    \partial(b_c) =\pi^*\overline{c}.
  \end{equation*}
\end{lemma}
\begin{proof}
  This is an elementary and direct consequence of the definition of $b_c$ and
  the simplicial coboundary: given $\sigma\colon \Delta_k\to PX$, the formula
  for $\boundary(b_c)(\sigma)=b_c(\boundary\sigma)$ is a sum over the
  different boundary components of $\sigma$ and the different summands in
  $b_c$ coming from the decomposition of $\Delta_{k-1}\times [0,1]$ into
  k-simplices. In this sum, one obtains ``interior'' contributions of
  k-simplices mapped to the interior of $\Delta_k\times [0,1]$. These occur in
  pairs which cancel out. Additionally, one obtains a top and a bottom
  contribution, corresponding to $\Delta_k\times \{1\}$ and $\Delta_k\times
  \{0\}$ giving rise to the last two summands in Equation
  \eqref{eq:boundary_and_b}. Finally, one obtains contributions of k-simplices
  mapped to $(\boundary(\Delta_k) )\times [0,1]$ which result precisely in the
  first summand. Writing out the elementary formulas is left to the reader.
\end{proof}

\begin{proposition}\label{prop:vanish_ac}
  Let $X$ be a topological space, $A$ an abelian coefficient group and  $c\in
  C^2_{\sing}(X;A)$ a cocycle representing an atoroidal cohomology
  class. Recall that this means  that $f^*[c]=0\in H^2(T^2;A)$ for every continuous map $f\colon
  T^2\to X$.

  Then, for $a_c:=r_1^*b_c-r_2^*b_c\in C^1_{\sing}(LX;A)$ as in
  Proposition \ref{prop:MV_cocycle} we have
  \begin{equation}\label{eq:ac_zero}
    [a_c]=0 \in H^1(LX;A).
  \end{equation}
  In particular, then also $\pi_2^*\overline{c} = 0 \in H^2_{\sing}(P_2X;A)$
  and $\overline{c}\in H^2(X\times X;A)$ has TC-weight $\wgt(\overline{c})=2$.
\end{proposition}
\begin{proof}
  The last statement follows from \eqref{eq:ac_zero}, namely that the cohomology class of $a_c$
  vanishes, together with Propositions \ref{prop:MV_cocycle} and
  \ref{prop:TC_crit}.

  It remains to show that $[a_c]=0\in H^1_{\sing}(LX;A)$. Now, in degree $1$
  the universal coefficient theorem gives the isomorphism
  \begin{equation*}
      H_{\sing}^1(LX;A) \iso \Hom(H^{sing}_1(X;\integers),A).
  \end{equation*}
  By the Hurewicz
  homomorphism it therefore suffices to show that
  \begin{equation*}
    a_c(\sigma\colon
    [0,1]\to LX)=0 \quad\forall \sigma\colon [0,1]\to LX \text{ with }
    \sigma(0)=\sigma(1),
  \end{equation*}
  i.e.~for the evaluation of $a_c$ on the image of the
  fundamental class of $S^1$ under the map $\overline \sigma\colon S^1\to
  LX$ defined by $\sigma$. Here and later, we use the identifications $[0,1]\iso\Delta_1$ and on
  the other hand $S^1= [0,1]/\sim$ with the equivalence relation $\sim$ generated by $0\sim 1$.

  To conveniently write down a formula for $a_c(\sigma)$ we apply the tensor-hom adjunction and define
  \begin{equation*}
    \Psi(\sigma)\colon T^2=S^1\times S^1\to X; (t,s)\mapsto \overline \sigma(t)(s).
  \end{equation*}

  Now we obtain in a straightforward way from the  definitions of our maps that
  \begin{equation}\label{eq:compute_ac}
    a_c(\sigma\colon [0,1]\to LX) =(\Psi(\sigma)^*c)(\tau)
  \end{equation}
  where $\tau\in C_2^{\sing}(T^2)$ is the signed sum of four singular simplices
  \begin{equation*}
    \tau = \sum_{j=1}^2\sum_{k=0}^1 (-1)^{j+1+k} (\sigma_{j,k}\colon\Delta_2\to T^2).
  \end{equation*}
   Here   $   \sigma_{j,k}\colon \Delta_2\to [0,1]/\sim\times[0,1]/\sim$
  is the affine linear map defined by sending the vertices $(v_0,v_1,v_2)$ of
  $\Delta_2$ to the following image points:
  \begin{center}
    \begin{tabular}{c|llll}
      $(j,k)$ & $(1,0)$& $(1,1)$& $(2,0)$& $(2,1)$\\\hline
      $v_0$ & $(0,0)$ & $(0,0)$ & $(0,1)$ & $(0,1)$\\
      $v_1$ & $(0,\frac{1}{2})$& $(1,0)$ & $(0,\frac{1}{2})$ & $(1,1)$\\
      $v_2$& $(1,\frac{1}{2})$ & $(1,\frac{1}{2})$ & $(1,\frac{1}{2})$& $(1,\frac{1}{2})$\\
    \end{tabular}
  \end{center}

    It is straightforward to see that $\tau$ is a singular cycle, indeed it is
    a fundamental cycle of $T^2$.

    By assumption, the pullback of the cohomology class represented by $c$
    along any map from $T^2$ to $X$ vanishes, in particular also
    $[\Psi(\sigma)^*c]=0\in H^2_{\sing}(T^2;A)$, and hence
    \begin{equation*}
      (\Psi^*(\sigma)(c))(\tau)=0,
    \end{equation*}
    as $\tau$ is a singular cycle. But by equation \eqref{eq:compute_ac} this
    is precisely what we have to establish to conclude the proof of
    Proposition \ref{prop:vanish_ac}.
\end{proof}

\section{Proof of the Main Theorems}

At this point, the proof of our main Theorems \ref{theo:main} and \ref{theo:specialmain} is
rather straightforward and follows the same lines as \cite[Section 4]{GrantMescher}. 

\begin{proof}[Proof of Theorem \ref{theo:main}]
  Assume we have the connected CW-complex $X$ of dimension $2n$, the coefficient field
  $F$ and the atoroidal cohomology class $u\in H^2(X;F)$ with $u^n\ne 0$.

  Form $\overline u:=\pi_1^*u-\pi_2^*u\in H^2(X\times X;F)$. Consider
  \begin{equation*}
    \overline u^{2n}\in H^{4n}(X\times X;F)\iso \bigoplus_{j=0}^{4n}
    H^j(X;F)\tensor H^{2n-j}(X;F).
  \end{equation*}
  Then the component of
  $\overline u^{2n}$ in the summand $H^{2n}(X;F)\tensor H^{2n}(X;F)$ is
  \begin{equation*}
    (-1)^n \binom{2n}{n} u^n \tensor u^n.
  \end{equation*}
  Now by assumption $u^n\ne 0$ and if $\characteristic(F)>2n$ or $\characteristic(F)=0$ also
  $\binom{2n}{n}\ne 0$ and consequently also $\overline u^n\ne 0$.

  By Proposition \ref{prop:vanish_ac}, $\wgt(\overline u)\ge 2$ and by Theorem
  \ref{theo:wgt_rules} $\wgt(\overline u^{2n})\ge 2n$ and therefore $\TC(X)\ge 4n$.
  Hence, $\TC(X)=4n$.
\end{proof}

\begin{proof}[Proof of Theorem \ref{theo:specialmain}] 
   Given the connected CW-complex $X$ of dimension $2n$, the coefficient field
   $F$ and the atoroidal
   cohomology classes $u_1,\dots,u_n\in H^2(X;F)$ such that $u_j^2=0$ and
   $u_1\cup\dots\cup u_n\ne 0\in H^{2n}(X;F)$, consider
   \begin{equation*}
     \overline u_j:=
     \pi_1^*u_j-\pi_2u_j\in H^2(X\times X;F)
   \end{equation*}
   and
   \begin{equation*}
     \begin{split}
       \overline u_1^2\cup\dots\cup
       \overline u_n^2 =&  (-2u_1\tensor u_1) \cup \dots \cup (-2u_n\tensor u_n)
       \\
       =&
         (-2)^n (u_1 \cup \dots \cup u_n)\tensor (u_1 \cup \dots \cup u_n)\\
       & \in H^{2n}(X;F)\tensor
       H^2(X;F) \subset H^{4n}(X\times X;F),
     \end{split}
   \end{equation*}
   where we use that $u_j^2=0$ for $j=1,\dots,n$. By assumption on
$\characteristic(F)\ne 2$ and $u_1 \cup \dots \cup u_n\ne 0$, we conclude that
   $\overline u_1^2 \cup \dots \cup \overline u_n^2\ne 0\in H^{4n}(X\times X;F)$.

   By Proposition \ref{prop:vanish_ac}, $\wgt(\overline u_j)\ge 2$ for $j\in
   \{1,\dots,n\}$ and by Theorem 
  \ref{theo:wgt_rules} $\wgt(\overline u_1^{2}\cup\dots\cup \overline u_n^2)
  \ge 2n$ and $\TC(X)\ge 4n$. Hence, $\TC(X)=4n$.
\end{proof}

\section{Applications and Examples}

\begin{definition}
  Let $X$ be a connected CW-complex. We call $X$ \emph{$2$-aspherical} if its
  universal covering $\tilde X$ is $2$-connected, i.e.~if $\pi_2(X)=0$.

  We call a cohomology class $u\in H^k(X;A)$ \emph{aspherical} if for every
  continuous map $f\colon S^k\to X$ the pullback satisfies $f^*u=0$. 
\end{definition}

Note that because of the existence of a degree one map $T^k\to S^k$ (in
particular, such that the induced map in degree $k$ cohomology is injective),
every atoroidal cohomology class is also aspherical. And observe that on a
$2$-aspherical space every degree $2$ cohomology class is aspherical.

We have the following partial converse of the last property:
\begin{proposition}\label{prop:automatic_atoroidal}
  Let $X$ be a connected CW-complex and let $u\in H^2(X;F)$ be an aspherical
  cohomology class, where $\characteristic(F)=0$. If $\pi_1(X)$ does not
  contain a subgroup
  isomorphic to $\pi_1(T^2)\iso \integers^2$ then $u$ is also atoroidal.

If $\pi_1(X)$ is torsion-free and does not contain a subgroup isomorphic to
$\integers^2$, the same conclusion also holds for an arbitrary coefficient
field $F$.
\end{proposition}
\begin{proof}
  By the proof\footnote{To be precise, the proof of \cite[Theorem 3]{Borat}
    has a gap: it only discusses the case where the kernel of $\pi_1(f)$
    contains a primitive element (i.e.~an element not divisible by any
    positive integer). This is automatically the case if $\pi_1(X)$ is
    torsion-free and the kernel is non-trivial. The general case follows by replacing $f$
    with the composition $f\circ p$ with a suitable finite covering projection
    $p\colon T^2\to T^2$ such that the kernel of $\pi_1(f\circ p)$ indeed
    contains a primitive element. It is here that we need
    $\characteristic(F)=0$: this implies that 
  the induced map $p^*\colon H^2(T^2;F)\to H^2(T^2;F)$ is injective.} of \cite[Theorem
  3]{Borat} if $f\colon T^2\to X$ is not injective, then under the conditions
  we have imposed  it holds that $f^*u=0$.
\end{proof}

\begin{proposition}\label{prop:2-aspherical_pi1}
  Let $X$ be a connected $2n$-dimensional $2$-aspherical CW-complex with
  cohomology 
  class $u\in H^2(X;F)$ such that $u^n\ne 0\in H^{2n}(X;F)$ for $F$ a
  coefficient field.

Assume that one of the following conditions is satisfied:
\begin{itemize}
\item $\characteristic(F)=0 $ and $\pi_1(X)$ does not contain a subgroup isomorphic to
  $\integers^2$, for example $\pi_1(X)$ is Gromov hyperbolic;
\item   $\characteristic(F)>2n$ and $\pi_1(X)$ is torsion free and does not contain a subgroup isomorphic to
  $\integers^2$

\end{itemize}

  Alternatively, assume that $\characteristic(F)\ne 2$ and that we have classes
  $u_1,\dots,u_n\in H^2(X;F)$ with $u_i^2=0$ and $u_1\cup\dots\cup u_n\ne 0
  \in H^{2n}(X;F)$.

    Then $\TC(X)=4n$.

\smallskip

  Without any assumption on the dimension of $X$, we can at least conclude that
  $\TC(X)\ge 4n$.
\end{proposition}
\begin{proof}
  Because $X$ is assumed to be $2$-aspherical, all classes $u$ and $u_i$ are
  aspherical. By Proposition \ref{prop:automatic_atoroidal} and the additional
  assumptions made, they are then also
  atoroidal. If $\dim(X)=4n$, the claim follows then immediately from our main Theorems
  \ref{theo:main} and \ref{theo:specialmain}. If $\dim(X)>4n$, we use the
  $4n$-dimensional $4n$-skeleton $X^{(4n)}\subset X$ and the injection
  $H^j(X;F)\to H^j(X^{(4n)};F)$ for $j\le 4n$ to conclude that
  $\TC(X^{(4n)})=4n$ and then by monotonicity $\TC(X)\ge \TC(X^{(4n)})=4n$. 
\end{proof}

\begin{example}
  Let $X$ be a $2n$-dimensional CW-complex with a CAT(-1)-metric, $F$ a
  coefficient field with $\characteristic(F)>2n$ and a class
  $u\in H^2(X;F)$ with $u^n\ne 0\in H^{2n}(X;F)$.

  Then $\TC(X) =4n$.
\end{example}
\begin{proof}
  If $X_i$ has a $CAT(-1)$ metric it is aspherical and hence, being finite
  dimensional, its fundamental 
  group is torsion-free and it is $2$-aspherical. In addition, because of
  strict negative curvature $\pi_1(X)$ is Gromov hyperbolic. We are therefore in a
  special case of Proposition \ref{prop:2-aspherical_pi1}.
\end{proof}

\begin{example}
  For $i=1,\dots,n$ let $X_i$ be an aspherical $2$-dimensional simplicial complex with a
  class $v_i\in H^2(X_i;F)$ for some coefficient field $F$ of characteristic
  $\characteristic(F)\ne 2$ and such
  that $v_i\ne 0\in H^{2}(X_i;F)$, and assume that $\pi_1(X_i)$ does not
  contain a subgroup isomorphic to $\integers^2$.
  Setting $X:=X_1\times\dots\times X_n$ we have $\TC(X)=4n$.

  \smallskip
 For a second class of examples, let $X$ satisfy the cohomological conditions
 as before, but without the assumption that the
 $X_i$ are aspherical and that $\pi_1(X_i)$ does not contain
 $\integers^2$. Assume that $f\colon Y\to X$ is a 
 strict  hyperbolization as in \cite{CharneyDavis}. Then $\TC(Y)=4n$.
\end{example}
\begin{proof}
  For $i=1,\dots,n$, the fatct that $X_i$ is aspherical and of finite
  dimension implies that it is
  $2$-aspherical and that $\pi_1(X_i)$ is torsion-free. We are therefore essentially in a 
  special case of Proposition \ref{prop:2-aspherical_pi1}. More precisely,
  setting $u_i:=\pr_i^*v_i$ for the projection $\pi_i\colon X\to X_i$ the
  classes $u_i$ are atoroidal as pullbacks of atoroidal classes and by the
  K\"unneth formula $u_1\cup\dots\cup u_n\ne 0 \in H^{2n}(X;F)$. The statement
  then follows from Theorem \ref{theo:specialmain}.

  For the second class of examples observe that, by definition, a strict
  hyperbolization is in particular a map $f\colon Y\to 
  X$ such that $Y$ is a simplicial complex with $\dim(X)=\dim(Y)$ which admits
  a $CAT(-1)$ metric and such that $f^*\colon H^*(Y;F)\to H^*(X;F)$ is
  injective\footnote{The formulation in \cite[Section 2]{CharneyDavis} uses
    the dual statement (5) of surjectivity in homology, but as observed in
    \cite[Section 2]{CharneyDavis}, this homological statement is a direct
    consequence of the conditions (1)--(4), and one sees immediately that they
    also imply injectivity in cohomology, and this with arbitrary even local
    coefficients}. As a consequence, $Y$ satisfies the asphericity and cohomological
  properties of Proposition \ref{prop:2-aspherical_pi1} and therefore $\TC(Y)=4n$.
\end{proof}

Concrete special cases of this example are given as follows:
\begin{example}

  We consider here the following types of CW-complexes:
  \begin{itemize}
  \item Let $\Gamma=\langle x_1,\dots,x_n\mid r\rangle$ be a $1$-relator group
    such that the relation $r$ is not a proper 
    power. Then the presentation complex $X$ is an aspherical 2-dimensional
    simplicial complex (compare \cite{Lyndon} and \cite[Section
    5]{ChiswellCollinsHuebschmann}). The relation $r$ is a word in the free
    group generated by $x_1,\dots, x_n$. If the image of $r$ in the
    abelianization of this group is zero (i.e.~the exponent sum in $r$ of
    $x_i$ is zero for each $i$), then $H^2(X;\integers)\iso\integers$. If for
    some $p>2$ the exponent sum in $R$ of each $x_i$ is divisible by $p$, then
    $H^2(X; \integers/p\integers)\iso \integers/p\integers$. In each case, a
    generator of this second cohomology is given by a class dual to the single
    $2$-cell in $X$. In addition, we assume that $\Gamma$ does not contain a
    subgroup isomorphic to $\integers^2$. Observe from the examples in
    \cite{Moldavanskii} that it is quite hard to decide from the relation
    whether this is the case or not.
  \item Let $M$ be an oriented compact connected prime atoroidal 3-manifold
    with non-empty boundary homeomorphic to $T^2$. Assume moreover that
    $\pi_2(M)=0$ and $\pi_1(M)$ is torsion 
    free. For example, let  $M$ be the complement of a hyperbolic knot in
    a closed 3-manifold $Y$ (i.e.~such that this 
    complement admits a complete Riemannian metric of constant sectional
    curvature $-1$ and finite volume). 

    Let $X$ be the 2-skeleton of $M$ for a CW-decomposition of $M$ (which is
    homotopy equivalent of $M$). Then $X$ is aspherical by \cite[Section
    5]{ChiswellCollinsHuebschmann} and every class in $H^2(X;F)$ is atoroidal. 
  \end{itemize}

  Pick now $X_1,\dots,X_n$ such that each $X_i$ is either a 1-relator presentation
  2-complex or a 3-manifold as above. Let $F$ be a field with
  $\characteristic(F)\ne 2$ and assume that 
  $H^2(X_i;F)\ne 0$ for each $i=1,\dots,n$
  (for the 1-relator 2-complexes: this means that the relation $r$ is zero
  modulo $\characteristic(F)$ in the abelianization of the group).

  Then $X:=X_1\times\dots\times X_n$ satisfies $TC(X)=4n$.

  We remark that the condition for the 1-relator presentations is very easy and
  convenient to check, in particular, we do not have to investigate the
  cup-product structure, but just compute the second homology additively.
\end{example}
\begin{proof}
  The only assertion to prove is that for the 3-manifolds as specified every
  class in $H^2(M;F)$ is atoroidal. For this, we observe that by the torus
  theorem for the class of
  3-manifolds considered every injective map $\integers^2\to \pi_1(M)$ is up
  to conjugation (which induces the identity map in cohomology) induced by a
  map $T^2\to \boundary M$, compare \cite[Section 7,
  (K.5)]{AschenbrennerWilton}.

The long exact cohomology sequence of the pair $(M,\boundary M)$ gives us
\begin{equation*}
  \begin{CD}
    H^2(M;F) @>{\iota^*}>> H^2(\boundary M;F) @>>> H^3(M,\boundary M;F) @>>> 0\\
    @VV=V @VV{\iso}V @VV{\iso}V @VV=V\\
    H^2(M;F) @>>> F @>{\iso}>> F @>>> 0
  \end{CD}
\end{equation*}
It follows from exactness that the restriction map $\iota^*\colon H^2(M;F)\to
H^2(\boundary M;F)$ is the zero map.

Now, by the torus theorem, we conclude that
\begin{itemize}
\item either $f\colon T^2\to M$ induces a non-injective map on fundamental
  groups and hence factors through $S^1$ so that $f^*u=0$ for all
  $u\in H^2(M;F)$
\item or $f\colon T^2\to M$ is homotopic to a map with image in $\boundary M$
  and hence again $f^*u=0$ for all $u\in H^2(M:F)$.
\end{itemize}
\end{proof}


\begin{bibdiv}
  \begin{biblist}
    % amsbib entries; as from mathscinet

\bib{AschenbrennerWilton}{book}{
   author={Aschenbrenner, Matthias},
   author={Friedl, Stefan},
   author={Wilton, Henry},
   title={3-manifold groups},
   series={EMS Series of Lectures in Mathematics},
   publisher={European Mathematical Society (EMS), Z\"{u}rich},
   date={2015},
   pages={xiv+215},
   isbn={978-3-03719-154-5},
   review={\MR{3444187}},
   doi={10.4171/154},
}
\bib{Borat}{article}{
   author={Borat, A.},
   title={Symplectically aspherical manifolds with nontrivial flux groups},
   journal={Acta Math. Hungar.},
   volume={149},
   date={2016},
   number={2},
   pages={523--525},
   issn={0236-5294},
   review={\MR{3518652}},
   doi={10.1007/s10474-016-0601-6},
 }
 \bib{CharneyDavis}{article}{
   author={Charney, Ruth M.},
   author={Davis, Michael W.},
   title={Strict hyperbolization},
   journal={Topology},
   volume={34},
   date={1995},
   number={2},
   pages={329--350},
   issn={0040-9383},
   review={\MR{1318879}},
   doi={10.1016/0040-9383(94)00027-I},
 }
\bib{ChiswellCollinsHuebschmann}{article}{
   author={Chiswell, Ian M.},
   author={Collins, Donald J.},
   author={Huebschmann, Johannes},
   title={Aspherical group presentations},
   journal={Math. Z.},
   volume={178},
   date={1981},
   number={1},
   pages={1--36},
   issn={0025-5874},
   review={\MR{627092}},
   doi={10.1007/BF01218369},
}
\bib{CohenVandembroucq}{article}{
   author={Cohen, Daniel C.},
   author={Vandembroucq, Lucile},
   title={Topological complexity of the Klein bottle},
   journal={J. Appl. Comput. Topol.},
   volume={1},
   date={2017},
   number={2},
   pages={199--213},
   issn={2367-1726},
   review={\MR{3975552}},
   doi={10.1007/s41468-017-0002-0},
 }
 \bib{DavisJanuszkiewicz}{article}{
   author={Davis, Michael W.},
   author={Januszkiewicz, Tadeusz},
   title={Hyperbolization of polyhedra},
   journal={J. Differential Geom.},
   volume={34},
   date={1991},
   number={2},
   pages={347--388},
   issn={0022-040X},
   review={\MR{1131435}},
}
\bib{Farber}{article}{
   author={Farber, Michael},
   title={Topological complexity of motion planning},
   journal={Discrete Comput. Geom.},
   volume={29},
   date={2003},
   number={2},
   pages={211--221},
   issn={0179-5376},
   review={\MR{1957228}},
   doi={10.1007/s00454-002-0760-9},
}
	
\bib{Farber_instabilities}{article}{
   author={Farber, Michael},
   title={Instabilities of robot motion},
   journal={Topology Appl.},
   volume={140},
   date={2004},
   number={2-3},
   pages={245--266},
   issn={0166-8641},
   review={\MR{2074919}},
   doi={10.1016/j.topol.2003.07.011},
 }
 \bib{FarberGrant}{article}{
   author={Farber, Michael},
   author={Grant, Mark},
   title={Symmetric motion planning},
   conference={
      title={Topology and robotics},
   },
   book={
      series={Contemp. Math.},
      volume={438},
      publisher={Amer. Math. Soc., Providence, RI},
   },
   date={2007},
   pages={85--104},
   review={\MR{2359031}},
   doi={10.1090/conm/438/08447},
}
    \bib{GrantMescher}{article}{
   author={Grant, Mark},
   author={Mescher, Stephan},
   title={Topological complexity of symplectic manifolds},
   journal={Math. Z.},
   volume={295},
   date={2020},
   number={1-2},
   pages={667--679},
   issn={0025-5874},
   review={\MR{4100027}},
   doi={10.1007/s00209-019-02366-x},
}
\bib{IwaseSakaiTsutaya}{article}{
   author={Iwase, Norio},
   author={Sakai, Michihiro},
   author={Tsutaya, Mitsunobu},
   title={A short proof for ${\rm tc}(K)=4$},
   journal={Topology Appl.},
   volume={264},
   date={2019},
   pages={167--174},
   issn={0166-8641},
   review={\MR{3975098}},
   doi={10.1016/j.topol.2019.06.014},
 }
 \bib{Lyndon}{article}{
   author={Lyndon, Roger C.},
   title={Cohomology theory of groups with a single defining relation},
   journal={Ann. of Math. (2)},
   volume={52},
   date={1950},
   pages={650--665},
   issn={0003-486X},
   review={\MR{47046}},
   doi={10.2307/1969440},
 }
 \bib{Moldavanskii}{article}{
   author={Moldavanski\u{\i}, D. I.},
   title={Certain subgroups of groups with one defining relation},
   language={Russian},
   journal={Sibirsk. Mat. \v{Z}.},
   volume={8},
   date={1967},
   pages={1370--1384},
   issn={0037-4474},
   review={\MR{220810}},
}
	
\bib{Sandrock}{thesis}{
  author={Sandrock, Luca},
  title={Topological complexity of symplectic CW-complexes},
  school={Universit\"at G\"ottingen},
  type={Bachelor thesis},
  date={2024},
  }
  \end{biblist}
\end{bibdiv}
\end{document}